\documentclass{article}
\usepackage[T1]{fontenc}
\usepackage{a4wide}
\usepackage{lipsum}
\usepackage{hyperref}

\title{On 3-Connected Cubic Planar Graphs and their Strong Embeddings on Orientable Surfaces}

\author{
Meike Weiß\thanks{Chair of Algebra and Representation Theory, RWTH Aachen University, Germany. E-mail: {\tt weiss@art.rwth-aachen.de}.}
\and
Reymond Akpanya\thanks{Chair of Algebra and Representation Theory, RWTH Aachen University, Germany. E-mail: {\tt akpanya@art.rwth-aachen.de}.}
\thanks{School of Mathematics and Statistics, The University of Sydney, Carslaw Building F07,
Camperdown NSW 2006, Australia. E-mail: {\tt reymond.akpanya@sydney.edu.au}.}
\and
Alice C.\ Niemeyer\thanks{Chair of Algebra and Representation Theory, RWTH Aachen University, Germany. E-mail: {\tt alice.niemeyer@art.rwth-aachen.de}.}
}

\date{}

\makeatletter

\usepackage{amsthm,amsmath,amssymb}
\usepackage{float}
\usepackage{cleveref}

\usepackage{cleveref}
\DeclareMathOperator{\Aut}{Aut}

\newtheorem{theorem}{Theorem}[section]
\newtheorem*{theorem*}{Theorem}
\newtheorem{lemma}[theorem]{Lemma}
\newtheorem{definition}[theorem]{Definition}

\newtheorem{corollary}[theorem]{Corollary}
\newtheorem*{corollary*}{Corollary}
\newtheorem*{conjecture*}{Conjecture}

\newtheorem{remark}[theorem]{Remark}
\newtheorem{proposition}[theorem]{Proposition}
\newtheorem*{question*}{Question}
\usepackage{subcaption}
\DeclareMathOperator{\C4C}{C4C}
\usepackage{amsmath, amssymb, amsfonts} 

\usepackage{ifthen}

\usepackage{tikz}
\usetikzlibrary{calc}
\usetikzlibrary{positioning}
\usetikzlibrary{shapes}
\usetikzlibrary{patterns}

\makeatletter
\edef\texforht{TT\noexpand\fi
  \@ifpackageloaded{tex4ht}
    {\noexpand\iftrue}
    {\noexpand\iffalse}}
\makeatother

\makeatletter
\newif\iftikz@node@phantom
\tikzset{
  phantom/.is if=tikz@node@phantom,
  text/.code=%
    \edef\tikz@temp{#1}%
    \ifx\tikz@temp\tikz@nonetext
      \tikz@node@phantomtrue
    \else
      \tikz@node@phantomfalse
      \let\tikz@textcolor\tikz@temp
    \fi
}
\usepackage{etoolbox}
\patchcmd\tikz@fig@continue{\tikz@node@transformations}{%
  \iftikz@node@phantom
    \setbox\pgfnodeparttextbox\hbox{}
  \fi\tikz@node@transformations}{}{}
\makeatother

\newcommand{\tikzAngleOfLine}{\tikz@AngleOfLine}
\def\tikz@AngleOfLine(#1)(#2)#3{%
  \pgfmathanglebetweenpoints{%
    \pgfpointanchor{#1}{center}}{%
    \pgfpointanchor{#2}{center}}
  \pgfmathsetmacro{#3}{\pgfmathresult}%
}

\tikzset{ 
    vertexNodePlain/.style = {fill=#1, shape=circle, inner sep=0pt, minimum size=2pt, text=none},
    vertexNodePlain/.default=gray,
    vertexPlain/labels/.style = {
        vertexNode/.style={vertexNodePlain=##1},
        vertexLabel/.style={gray}
    },
    vertexPlain/nolabels/.style = {
        vertexNode/.style={vertexNodePlain=##1},
        vertexLabel/.style={text=none}
    },
    vertexPlain/.style = vertexPlain/#1,
    vertexPlain/.default=labels
}
\tikzset{
    vertexNodeNormal/.style = {fill=#1, shape=circle, inner sep=0pt, minimum size=4pt, text=none},
    vertexNodeNormal/.default = blue,
    vertexNormal/labels/.style = {
        vertexNode/.style={vertexNodeNormal=##1},
        vertexLabel/.style={blue}
    },
    vertexNormal/nolabels/.style = {
        vertexNode/.style={vertexNodeNormal=##1},
        vertexLabel/.style={text=none}
    },
    vertexNormal/.style = vertexNormal/#1,
    vertexNormal/.default=labels
}
\tikzset{
    vertexNodeBallShading/pdf/.style = {ball color=#1},
    vertexNodeBallShading/svg/.style = {fill=#1},
    vertexNodeBallShading/.code = {
        \if\texforht
            \tikzset{vertexNodeBallShading/svg=#1!90!black}
        \else
            \tikzset{vertexNodeBallShading/pdf=#1}
        \fi
    },
    vertexNodeBall/.style = {shape=circle, vertexNodeBallShading=#1, inner sep=2pt, outer sep=0pt, minimum size=3pt, font=\tiny},
    vertexNodeBall/.default = white,
    vertexBall/labels/.style = {
        vertexNode/.style={vertexNodeBall=##1, text=black},
        vertexLabel/.style={text=none}
    },
    vertexBall/nolabels/.style = {
        vertexNode/.style={vertexNodeBall=##1, text=none},
        vertexLabel/.style={text=none}
    },
    vertexBall/.style = vertexBall/#1,
    vertexBall/.default=labels
}
\tikzset{ 
    vertexStyle/.style={vertexNormal=#1},
    vertexStyle/.default = labels
}

\newcommand{\vertexLabelR}[4][]{
    \ifthenelse{ \equal{#1}{} }
        { \node[vertexNode] at (#2) {#4}; }
        { \node[vertexNode=#1] at (#2) {#4}; }
    \node[vertexLabel, #3] at (#2) {#4};
}
\newcommand{\vertexLabelA}[4][]{
    \ifthenelse{ \equal{#1}{} }
        { \node[vertexNode] at (#2) {#4}; }
        { \node[vertexNode=#1] at (#2) {#4}; }
    \node[vertexLabel] at (#3) {#4};
}

\newcommand{\edgeLabelColor}{blue!20!white}
\tikzset{
    edgeLineNone/.style = {draw=none},
    edgeLineNone/.default=black,
    edgeNone/labels/.style = {
        edge/.style = {edgeLineNone=##1},
        edgeLabel/.style = {fill=\edgeLabelColor,font=\small}
    },
    edgeNone/nolabels/.style = {
        edge/.style = {edgeLineNone=##1},
        edgeLabel/.style = {text=none}
    },
    edgeNone/.style = edgeNone/#1,
    edgeNone/.default = labels
}
\tikzset{
    edgeLinePlain/.style={line join=round, draw=#1},
    edgeLinePlain/.default=black,
    edgePlain/labels/.style = {
        edge/.style={edgeLinePlain=##1},
        edgeLabel/.style={fill=\edgeLabelColor,font=\small}
    },
    edgePlain/nolabels/.style = {
        edge/.style={edgeLinePlain=##1},
        edgeLabel/.style={text=none}
    },
    edgePlain/.style = edgePlain/#1,
    edgePlain/.default = labels
}
\tikzset{
    edgeLineDouble/.style = {very thin, double=#1, double distance=.8pt, line join=round},
    edgeLineDouble/.default=gray!90!white,
    edgeDouble/labels/.style = {
        edge/.style = {edgeLineDouble=##1},
        edgeLabel/.style = {fill=\edgeLabelColor,font=\small}
    },
    edgeDouble/nolabels/.style = {
        edge/.style = {edgeLineDouble=##1},
        edgeLabel/.style = {text=none}
    },
    edgeDouble/.style = edgeDouble/#1,
    edgeDouble/.default = labels
}
\tikzset{
    edgeStyle/.style = {edgePlain=#1},
    edgeStyle/.default = labels
}

\newcommand{\faceColorY}{yellow!60!white}   
\newcommand{\faceColorB}{blue!60!white}     
\newcommand{\faceColorC}{cyan!60}           
\newcommand{\faceColorR}{red!60!white}      
\newcommand{\faceColorG}{green!60!white}    
\newcommand{\faceColorO}{orange!50!yellow!70!white} 

\newcommand{\faceColor}{\faceColorY}
\newcommand{\faceColorSwap}{\faceColorC}

\tikzset{
    face/.style = {fill=#1},
    face/.default = \faceColor,
    faceY/.style = {face=\faceColorY},
    faceB/.style = {face=\faceColorB},
    faceC/.style = {face=\faceColorC},
    faceR/.style = {face=\faceColorR},
    faceG/.style = {face=\faceColorG},
    faceO/.style = {face=\faceColorO}
}
\tikzset{
    faceStyle/labels/.style = {
        faceLabel/.style = {}
    },
    faceStyle/nolabels/.style = {
        faceLabel/.style = {text=none}
    },
    faceStyle/.style = faceStyle/#1,
    faceStyle/.default = labels
}
\tikzset{ face/.style={fill=#1} }
\tikzset{ faceSwap/.code=
    \ifdefined\swapColors
        \tikzset{face=\faceColorSwap}
    \else
        \tikzset{face=\faceColor}
    \fi
}

\usepackage{hyperref}

\begin{document}
\thispagestyle{empty}
\maketitle

\begin{abstract}
Although the strong embedding of a $3$-connected planar graph $G$ on the sphere is unique, $G$ can have different inequivalent strong embeddings on a surface of positive genus. 
If $G$ is cubic, then the strong embeddings of $G$ on the projective plane, the torus and the Klein bottle each are in one-to-one correspondence with certain subgraphs of the dual graph $G^\ast$. 
Here, we exploit this characterisation and show that two strong embeddings of $G$ on the projective plane, the torus or the Klein bottle are isomorphic if and only if the corresponding subgraphs of $G^{\ast}$ are contained in the same orbit under $\Aut(G^{\ast})$.
This allows us to construct a data base containing all isomorphism classes of strong embeddings on the projective plane, the torus and the Klein bottle of all $3$-connected cubic planar graphs with up to $22$ vertices. 
Moreover, we establish that cyclically 4-edge connected cubic planar graphs can be strongly embedded on orientable surfaces of positive genera. We use this to show that a 3-connected cubic planar graph has no strong embedding on orientable surfaces of positive genera if and only if it is the dual of an Apollonian network.
\end{abstract}
\section{Introduction}\label{section:Introduction}
To lay the foundation for our investigation, we adopt the standard terminology of topological graph theory, as presented in~\cite{TopologicalGraphTheory,GraphsOnSurfaces}.
A central aspect of topological graph theory is the study of graphs embedded on surfaces.
In this work, we focus on strong graph embeddings, i.e.\ embeddings such that a given graph is drawn on a surface so that every cell is bounded by a cycle.
A long-standing open problem that arises in the study of strong graph embeddings is the \textbf{oriented strong embedding conjecture}, see \cite{CycleDoubleCoverConjecture}.

\begin{conjecture*}  Every $2$-connected graph has a strong embedding on some orientable surface.
\end{conjecture*}
This conjecture has been the focus of various studies, see \cite{ELLINGHAM2011495,bojan,richter} for example.
We observe that the above conjecture naturally holds for \textbf{planar graphs}, i.e.\ graphs that can be embedded on the plane such that any two embedded edges intersect at most in their common endpoints.
Whitney’s celebrated Unique Embedding Theorem states that a strong embedding of a $3$-connected planar graph on the sphere is unique, see \cite{Whitney}.
However, it is possible that a given $3$-connected planar graph has different inequivalent and even non-isomorphic strong embeddings on a surface of positive genus. 
Building on the work of Enami \cite{EnamiEmbeddings} and {Weiß et al.~\cite{PaperMeikeStrong} this paper investigates non-isomorphic strong embeddings of 3-connected cubic planar graphs on surfaces of positive genera. In his work,
Enami characterises certain subgraphs of the dual graph $G^\ast$ for a $3$-connected cubic planar graph $G$ that are used to define inequivalent embeddings of $G$ on the projective plane, the torus or the Klein bottle.
Based on these results, Weiß et al.\ determine those subgraphs of $G^\ast$ that lead to inequivalent strong embeddings of $G$ on the projective plane, the torus or the Klein bottle.
This allows us to prove the following theorem.

 \begin{theorem*}
The non-isomorphic strong embeddings of a $3$-connected cubic planar graph $G$ on the projective plane, the torus and the Klein bottle are in a one-to-one correspondence with certain orbits of $\Aut(G^{\ast})$ acting on the subgraphs of $G^\ast$ (see \Cref{theorem:nonIsomorphic}).
\end{theorem*}
  We use our theoretical results to compute a data base containing exactly one representative from each isomorphism class of strong embeddings of $3$-connected cubic planar graphs with at most $22$ vertices on the projective plane, the torus and the Klein bottle, see \cite{dataReembeddings}. Moreover, we establish the following result.
\begin{theorem*}
    There exists a family of $3$-connected cubic planar graphs such that the number of non-isomorphic strong embeddings on the torus and the Klein bottle grows exponentially in the number of vertices (see \Cref{theorem:exponential}).
 \end{theorem*}
A class of $3$-connected cubic planar graphs that plays a central role in this work is the class of \textbf{Apollonian duals}, i.e.\ the dual graphs of Apollonian networks. Here, an \textbf{Apollonian network} is a planar graph that can be constructed from recursively subdividing the faces of the complete graph $K_4$ into three new faces, see ~\cite{Beinecke,Birkhoff_1930,pachner,fowler,grünbaum}.
\Cref{fig:exIntro} illustrates an example of an Apollonian network and its dual graph. We exploit these graphs to establish the following result.

\begin{theorem*}\label{theorem:main2}
The Apollonian duals are exactly the $3$-connected cubic planar graphs that have no strong embedding on any orientable surface of positive genus (see \Cref{theorem:noOrient}).
\end{theorem*}
\begin{figure}[H]
    \centering
        \begin{subfigure}{.45\textwidth}
        \centering
        \begin{tikzpicture}[scale=2.]
            \tikzset{knoten/.style={circle,fill=black,inner sep=0.6mm}}
            \node [knoten] (a) at (0.,0) {};
            \node [knoten] (b) at (1,0) {};
            \node [knoten] (c) at (0.5,0.8660) {};
            \node [knoten] (d) at (0.5,0.5*0.8660) {};
            \node [knoten] (e) at (barycentric cs:a=1,b=1,d=1) {};
            
            \draw[-,thick] (a) to (b);
            \draw[-,thick] (a) to (c);
            \draw[-,thick] (a) to (d);
            \draw[-,thick] (b) to (c);
            \draw[-,thick] (b) to (d);
            \draw[-,thick] (c) to (d);
            \draw[-,thick] (e) to (d);
            \draw[-,thick] (e) to (a);
            \draw[-,thick] (e) to (b);
        \end{tikzpicture}
        \caption{}
        \label{fig:a}
    \end{subfigure}
        \begin{subfigure}{.45\textwidth}
        \centering
        \begin{tikzpicture}[scale=2.]
            \tikzset{knoten/.style={circle,fill=black,inner sep=0.6mm}}
            \node [knoten] (a) at (0.,0) {};
            \node [knoten] (a1) at (0.25,-0.15) {};
            \node [knoten] (b) at (1,0) {};
            \node [knoten] (b1) at (0.75,-0.15) {};
            \node [knoten] (c) at (0.5,-0.8660) {};
            \node [knoten] (c1) at (0.5,-0.6) {};
            
            \draw[-,thick] (a) to (b);
            \draw[-,thick] (b) to (c);
            \draw[-,thick] (c) to (a);
            \draw[-,thick] (a1) to (b1);
            \draw[-,thick] (b1) to (c1);
            \draw[-,thick] (c1) to (a1);

            \draw[-,thick] (a) to (a1);
            \draw[-,thick] (b) to (b1);
            \draw[-,thick] (c) to (c1);
        \end{tikzpicture}
        \caption{}
        \label{fig:fgD}
    \end{subfigure}
    \caption{An Apollonian network (a) and its Apollonian dual (b)}
    \label{fig:exIntro}
\end{figure}
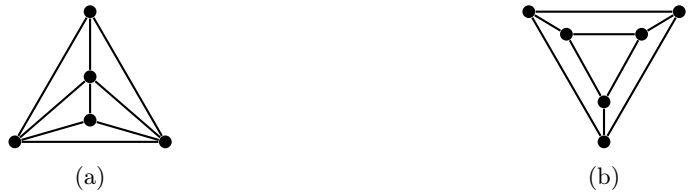

This result is based on the following theorem which we use to consider strong embeddings on orientable surfaces of positive genus.
\begin{theorem*}
    Let $G$ be a cubic planar graph. If $G$ is cyclically 4-edge connected, then $G$ has a strong embedding on an orientable surface (see \Cref{corollary:orientable}).
\end{theorem*}

Our paper is structured as follows: In \Cref{sec:preliminaries} we introduce preliminary notions on graphs and their embeddings that are necessary to conduct our investigations.
In \Cref{sec:isomorphism}, we show that cyclically 4-edge connected cubic planar graphs have a strong embedding on an orientable surface of positive genus.
Moreover, we characterise non-isomorphic strong embeddings of 3-connected cubic planar graphs on certain surfaces and show that the number of non-isomorphic strong embeddings on the torus and the Klein bottle can grow exponentially in the number of vertices.
Additionally, we provide details on the data base of non-isomorphic strong embeddings of all $3$-connected cubic planar graphs that we have computed during our investigation. Studying this database has motivated us to study strong embeddings of Apollonian duals, see \Cref{subsec:apollonian}. In this section we have been able to establish the existence of strong embeddings of Apollonian duals on the projective plane and the Klein bottle. As a main result we show in \Cref{sec:orient} that Apollonian duals are precisely the $3$-connected cubic planar graphs that do not admit a strong embedding on any orientable surface of positive genus.

We exploit the computer algebra system \textsc{GAP} \cite{GAP4} to conduct our research. In particular, we used the \textsc{GAP}-packages \textit{SimplicialSurfaces}~\cite{simplicialsurfacegap} and \textit{Digraphs}~\cite{Digraphs}. The package \textit{SimplicialSurfaces} contains functions that enable us to study the combinatorial structure of triangulations of surfaces and strong embeddings of cubic graphs. 
For our computations, we used the database of $3$-connected cubic planar that has been established in~\cite{brinkmann}.

\section{Preliminaries}\label{sec:preliminaries}
We assume that all the graphs in this paper are undirected, connected, simple and finite. Let $G$ be such a graph. An \textbf{automorphism} of $G$ is a bijective map $\varphi: V(G)\rightarrow V(G)$ satisfying $\bigl\{ \{\varphi(v),\varphi(w)\}\mid \{v,w\}\in E(G)\bigr\}=E(G)$. We denote the group of automorphisms of $G$ by $\Aut(G)$.
Furthermore, for $e=\{v,w\}\in E(G)$ and $\varphi\in \Aut(G)$ we define $\varphi(e):=\{\varphi(v),\varphi(w)\}\in E(G)$.  Hence, the orbit of a subgraph $H\subseteq G$ under $\Aut(G)$ is defined as follows:
$$H^{\Aut(G)}:=\bigl\{H'\mid V(H')=\varphi(V(H)) \text{ and } E(H')=\varphi(E(H))\text{ for }\varphi\in\Aut(G)\bigr\}.$$
A \textbf{cycle} in $G$ is a closed walk, where only the start and end vertices are the same. Furthermore, $G$ is \textbf{cyclically $k$-edge connected} if there is no set $A\subseteq E(G)$ of at most $k{-}1$ edges such that the graph $G\setminus A$ has at least two connected components containing a cycle. Note that a \textbf{cyclic edge cut} is a set of edges such that deleting them results in a graph with at least two connected components having a cycle.
A cycle of length three in a given graph is called a \textbf{triangle}. We call a triangle $T=(v_1,v_2,v_3)$ of a graph $G$ \textbf{separating} if $G\setminus\{v_1,v_2,v_3\}$ is disconnected.
In the following definition we describe how triangles of a graph can be subdivided into three triangles to obtain a new graph.
\begin{definition}\label{def:subdivision}
   Let $G$ be a graph and $T=(v_1,v_2,v_3)$ a triangle of $G$. We define $G^T$ as the graph that is constructed by adding a new vertex $v$ to $G$ such that $v$ is incident to $v_1,v_2$ and $v_3$. We call this operation the \textbf{\emph{subdivision of a triangle}}.
   The inverse operation, namely removing a vertex $w\in V(G)$ with $\deg(w)=3$ from $G$, is called the \textbf{\emph{deletion}} of $w$. The graph resulting from this operation is denoted by $G_w$.
\end{definition}
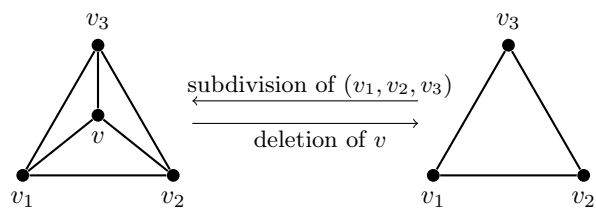
\begin{figure}[H]
\centering
\begin{minipage}{0.15\textwidth}
\begin{tikzpicture}[scale=2]
\tikzset{knoten/.style={circle,fill=black,inner sep=0.6mm}}

\node [knoten,label=below: $v_1$] (V1) at (0,0) {};
\node [knoten,label=below: $v_2$] (V2) at (1,0) {};
\node [knoten,label=$v_3$] (V3) at (0.5, 0.8660254037844386) {};
\node [knoten,label=below: $v$] (V4) at (0.5,0.4) {};

\draw[-,thick] (V2) to (V1);
\draw[-,thick] (V1) to (V3);
\draw[-,thick] (V1) to (V4);
\draw[-,thick] (V4) to (V2);
\draw[-,thick] (V3) to (V4);
\draw[-,thick] (V2) to (V3);

\end{tikzpicture}
\end{minipage}
\begin{minipage}{0.2\textwidth}
    \begin{tikzpicture}[scale=1]
        \draw[->] (-1.5,0) -- (1.5,0);
        \draw[<-] (-1.5,0.3) -- (1.5,0.3);
        \node (E) at (0.2,-0.2) {\small{deletion of $v$}};
        \node (F) at (0.2,0.5) {\small{subdivision of $(v_1,v_2,v_3)$}};
    \end{tikzpicture}
\end{minipage}
\begin{minipage}{0.15\textwidth}
\begin{tikzpicture}[scale=2]
\tikzset{knoten/.style={circle,fill=black,inner sep=0.6mm}}

\node [knoten,label=below:$v_1$] (V1) at (0,0) {};
\node [knoten,label=below:$v_2$] (V2) at (1,0) {};
\node [knoten,label=$v_3$] (V3) at (0.5, 0.8660254037844386) {};

\draw[-,thick] (V2) to (V1);
\draw[-,thick] (V1) to (V3);
\draw[,thick] (V3) to (V2);
\end{tikzpicture}
\end{minipage}
\caption{Subdivision of the triangle $(v_1,v_2,v_3)$ and deletion of the vertex $v$ of degree three}
\end{figure}

\begin{definition}
    An \emph{\textbf{Apollonian network}} is a planar graph that is constructed by recursively subdividing a triangular face starting from the complete graph $K_4$. Moreover, an \emph{\textbf{Apollonian dual}} is defined as the dual graph of an Apollonian network.
\end{definition}

An \textbf{embedding} of a graph $G$ on a compact 2-dimensional manifold $S$ without boundary is an injective and continuous map $\beta :G\rightarrow S$. The \textbf{cells} of $\beta(G)$ are the connected components of $S\setminus\beta(G)$. 
All embeddings in this paper are assumed to be cellular embeddings, i.e.\ the cells of a given embedding are homeomorphic to open $2$-cells.
A closed walk of the embedded graph $\beta(G)$ is called a \textbf{facial walk} if it bounds a cell. If a facial walk is a cycle, we call it a \textbf{facial cycle}.
If $G$ is a planar graph, we call the facial walks of the embedding on the sphere \textbf{faces}.
Furthermore, a \textbf{planar triangulation} is a planar graph with all faces being triangles, e.g.\ Apollonian networks.
The embedding $\beta$ of $G$ is called \textbf{strong} if all corresponding facial walks are facial cycles. 
Note that two strong embeddings of $G$ are \textbf{equivalent} if and only if the corresponding sets of facial cycles are equal. Additionally, two strong embeddings of $G$ are called \textbf{isomorphic} if and only if the corresponding sets of facial cycles are contained in the same orbit under the action of $\Aut(G)$. 
From a combinatorial perspective every embedding of $G$ can be defined by a rotation system and a set of twisted edges. For each vertex $v$ of the embedded graph $G$ a \textbf{rotation system} defines the cyclic ordering of the edges that are incident to $v$. 
Furthermore, the \textbf{twisted edges} give a partial combinatorial description of the edges that are included in a facial walk of the given embedding. For simplicity, we refer to \cite{TopologicalGraphTheory,GraphsOnSurfaces,PaperMeikeStrong} for a formal definition of twisted edges. Loosely speaking, the twisted edges of a graph define whether the rotation system at a vertex must be considered in clockwise or anticlockwise orientation. These edges can be encoded into a \textbf{signature}, which is a function $\lambda$ that maps the edges of a given graph $G$ onto the set $\{-1,1\}$. The edges $e\in E(G)$ satisfying $\lambda(e)=-1$ are exactly the twisted edges. 
A graph embedding is defined uniquely using both a rotation system and a signature. The facial walks of the embedded graph can be computed from a given rotation system and a given signature by applying the face traversal algorithm. More details on this algorithm, as well as on rotation systems and twisted edges, can be found in~\cite{TopologicalGraphTheory,GraphsOnSurfaces,PaperMeikeStrong}.

Strong embeddings of $3$-connected cubic planar graphs on the projective plane, the torus and the Klein bottle can be characterised by certain subgraphs of the corresponding dual graphs, see \cite{PaperMeikeStrong}. In order to recall this result, we denote the \textbf{complete graph} on $n$ vertices by $K_n$. For a positive integer $k\geq 2$ we further denote a \textbf{complete $k$-partite graph} with $k$ partition sets $V_1, V_2,\dots, V_k$ such that $\vert V_i\vert=n_i$ for $1\leq i \leq k$ by $K_{n_1,n_2,\dots, n_k}$.

\begin{theorem}[see \cite{PaperMeikeStrong}]\label{theorem:reembedding}
Let $G$ be a 3-connected cubic planar graph. There exists a one-to-one correspondence between inequivalent strong embeddings of $G$ on
    \begin{itemize}
        \item[1)] the projective plane and subgraphs of $G^{\ast}$ that are isomorphic to $K_4$.
        \item[2)] the torus and subgraphs of $G^{\ast}$ that are isomorphic to $K_{2,2,2}$ or isomorphic to $K_{2,2m}$ for $m\geq 1$, where for $K_{2,2m}$ the vertices in the partition sets of size two are not adjacent in $G^{\ast}$.
        \item[3)] the Klein bottle and subgraphs of $G^{\ast}$ that are isomorphic to $A_3, A_5, A_6$ or $K_{2,2m-1}$ for $m\geq 2$, where for $K_{2,2m-1}$ the vertices in the partition sets of size two are not adjacent in $G^{\ast}$.
    \end{itemize}
\end{theorem}
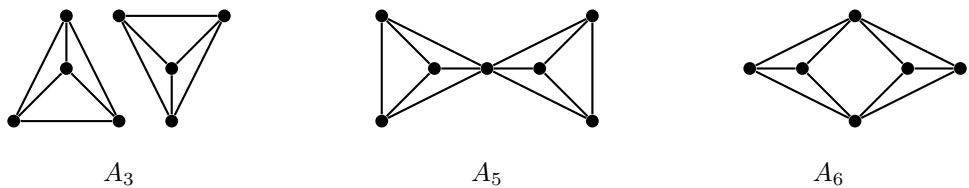
\begin{figure}[H]
    \centering
    \begin{minipage}{0.32\textwidth}
        \centering
        \begin{tikzpicture}[scale=1.4]
            \tikzset{knoten/.style={circle,fill=black,inner sep=0.6mm}}
            \node [knoten] (a) at (0,0) {};
            \node [knoten] (b) at (1,0) {};
            \node [knoten] (c) at (0.5,0.5) {};
            \node [knoten] (d) at (0.5,1) {};
            \node [knoten] (e) at (1,1) {};
            \node [knoten] (f) at (2,1) {};
            \node [knoten] (g) at (1.5,0.5) {};
            \node [knoten] (h) at (1.5,0) {};
            
            \draw[-, thick] (a) to (b);
            \draw[-, thick] (a) to (c);
            \draw[-, thick] (a) to (d);
            \draw[-, thick] (b) to (c);
            \draw[-, thick] (b) to (d);
            \draw[-, thick] (c) to (d);
            
            \draw[-, thick] (e) to (f);
            \draw[-, thick] (e) to (g);
            \draw[-, thick] (e) to (h);
            \draw[-, thick] (f) to (g);
            \draw[-, thick] (f) to (h);
            \draw[-, thick] (g) to (h);
            
            \node (C) at (1,-0.5) {$A_3$};
        \end{tikzpicture}
    \end{minipage}
    \begin{minipage}{0.32\textwidth}
        \centering
        \begin{tikzpicture}[scale=1.4]
            \tikzset{knoten/.style={circle,fill=black,inner sep=0.6mm}}
            \node [knoten] (a) at (0,0) {};
            \node [knoten] (b) at (0,1) {};
            \node [knoten] (c) at (0.5,0.5) {};
            \node [knoten] (d) at (1,0.5) {};
            \node [knoten] (e) at (1.5,0.5) {};
            \node [knoten] (f) at (2,1) {};
            \node [knoten] (g) at (2,0) {};
            
            \draw[-, thick] (a) to (b);
            \draw[-, thick] (a) to (c);
            \draw[-, thick] (a) to (d);
            \draw[-, thick] (b) to (c);
            \draw[-, thick] (b) to (d);
            \draw[-, thick] (c) to (d);
            
            \draw[-, thick] (e) to (f);
            \draw[-, thick] (e) to (g);
            \draw[-, thick] (e) to (d);
            \draw[-, thick] (f) to (g);
            \draw[-, thick] (f) to (d);
            \draw[-, thick] (g) to (d);
            
            \node (C) at (1,-0.5) {$A_5$};
        \end{tikzpicture}
    \end{minipage}
    \begin{minipage}{0.32\textwidth}
        \centering
        \begin{tikzpicture}[scale=1.4]
            \tikzset{knoten/.style={circle,fill=black,inner sep=0.6mm}}
            \node [knoten] (a) at (0,0) {};
            \node [knoten] (b) at (0.5,0) {};
            \node [knoten] (c) at (1,0.5) {};
            \node [knoten] (d) at (1,-0.5) {};
            \node [knoten] (e) at (1.5,0) {};
            \node [knoten] (f) at (2,0) {};
            
            \draw[-, thick] (a) to (b);
            \draw[-, thick] (a) to (c);
            \draw[-, thick] (a) to (d);
            \draw[-, thick] (b) to (c);
            \draw[-, thick] (b) to (d);
            \draw[-, thick] (c) to (e);
            \draw[-, thick] (c) to (f);
            \draw[-, thick] (d) to (e);
            \draw[-, thick] (d) to (f);
            \draw[-, thick] (e) to (f);
            
            \node (C) at (0.75,-1) {$A_6$};
        \end{tikzpicture}
    \end{minipage}
    \caption{The graphs $A_3, A_5$ and $A_6$}
    \label{fig:subgraphsKleinBottle}
\end{figure}

Next, we outline how an embedding of a 3-connected cubic planar graph $G$ can be described in terms of a subgraph of its dual.
First, we observe that the unique planar embedding of $G$ can be described by a rotation system together with a signature $\lambda$ that is $1$ for all edges of $G.$
Thus, it is sufficient to define other embeddings of $G$ by altering its corresponding signature, which is achieved by exploiting the edges of a certain subgraph $H$ of $G^{\ast}$. 
This subgraph gives rise to a signature $\lambda'$ that maps an edge $e\in E(G)$ onto $-1$ if and only if its dual edge is contained in $H$. In this case, we refer to $H$ as a \textbf{twisted subgraph} of $G^{\ast}$. 
Enami establishes that a given $3$-connected cubic planar graph has an embedding on an orientable surface if and only if the dual graph contains an even subgraph, i.e.\ a subgraph where all vertices have even degree, see \cite[Lemma 8]{EnamiEmbeddings}.
In \cite[Example 9]{PaperMeikeStrong} we can find an example of a strong embedding that has been computed by exploiting the above approach.

\section{Characterisation of Isomorphism Classes}\label{sec:isomorphism}
We start this section by showing that cyclically 4-edge connected cubic planar graphs admit strong embeddings on orientable surfaces of positive genera. Then, we characterise the isomorphic strong embeddings of a 3-connected cubic planar graph on the torus, the projective plane, and the Klein bottle.
Exploiting this characterisation allows us to compile a data base of non-isomorphic strong embeddings on the torus, the projective plane and the Klein bottle. Here, we present some of the details of this data base. Finally, we prove the existence of a family of 3-connected cubic planar graphs that exhibits exponential growth of non-isomorphic strong embeddings on the torus and the Klein bottle in the number of vertices.

\subsection{Orientable Strong Embeddings of Cyclically 4-Edge Connected Cubic Planar Graphs}
\Cref{theorem:reembedding} gives us an exact description of the twisted subgraphs that define strong embeddings of 3-connected cubic planar graphs on the torus. Hence, in order to prove that cyclically 4-edge connected cubic planar graphs can be strongly embedded on orientable surfaces we have to extend this result. In particular, we have to construct a suitable class of graphs such that at least one of them forms a twisted subgraph of an arbitrary cyclically $4$-edge connected cubic planar graph defining an orientable strong embedding. As a first step, we recall the definition of visited edges from \cite{PaperMeikeStrong}. This will help us to determine twisted subgraphs that define strong embeddings.

Let therefore $G$ be a 3-connected cubic planar graph and $H$ a twisted subgraph of $G^{\ast}$. 
We consider the embedding of $H$, denoted by $\beta_G(H)$, with the rotation system induced by $G$ and with all its edges treated as twisted. 
We use this embedding to decide whether the given twisted subgraph $H$ gives rise to a strong embedding of $G$.

\begin{definition}\label{def:dualcycles}
    Let $G$ be a $3$-connected cubic planar graph and $H$ a twisted subgraph of $G^{\ast}$.
    Furthermore, let $C=(w_1,\dots,w_n)$ be a facial walk of $\beta_G(H)$. Note that the orientations at the vertices defined by $C$ alternate, since all edges of $C$ are twisted.
    Let $e_1,\ldots,e_n$ be the edges of $H$ defined by $e_i=\{w_{i-1},w_i\}$ for $i\in\{2,\dots,n\}$ and $e_1=\{w_1,w_n\}$.
    We define the tuple of \emph{\textbf{visited edges}} at a vertex $w_i$ in $C$ to be the edges in $G^{\ast}$ that are incident to $w_i$ and lie between $e_i$ and $e_{i+1}$ in the clockwise or anticlockwise orientation depending on $C$.
\end{definition}
Weiß et al.\ show that the embedding $\beta$ of a 3-connected cubic planar graph $G$ defined by a twisted subgraph $H$ is strong if and only if for each facial walk $C$ of $\beta_G(H)$ and each pair of vertices in $C$ the set of visited edges are disjoint, see \cite[Theorem 19]{PaperMeikeStrong}.

To enhance clarity of visited edges of a given 3-connected cubic planar graph, we give the following illustration:
Let $G$ be a $3$-connected cubic planar graph such that its dual graph $G^{\ast}$ contains a subgraph $H$ that is isomorphic to $K_{2,2}$ with $V(H)=\{a,b,c,d\}.$ Furthermore, let $\{a,c\}$ and $\{b,d\}$ be the partition sets of $H$. The embedding $\beta_G(H)$, where all four edges are twisted, has exactly two facial walks $C_1,C_2$ that are both described by the cycle $(a,b,c,d)$. We observe that the orientation of a vertex $v\in V(H)$ is clockwise with respect to $C_1$ if and only if the orientation of $v$ with respect to $C_2$ is anti-clockwise. Moreover, we see that the orientations of the vertices $a,b,c$ and $d$ alternate in both facial walks. 
Without loss of generality, let $C_1$ be the facial walk where the orientation at $a$ and $c$ is clockwise and $C_2$ the facial walk where the orientation at $a$ and $c$ is anticlockwise. This allows us to illustrate the different sets of visited edges of $G$ in \Cref{fig:visitedEdges} where the visited edges of $C_1$ are indicated in orange and of $C_2$ in blue. Due to the result established by Weiß et al.\ the embedding of $G$ induced by the twisted subgraph $H$ is strong if and only if $a$ and $c$ are not adjacent in $G^\ast$ and neither are $b$ and $d$.

\begin{figure}[H]
    \centering
    \begin{tikzpicture}[scale=2]
            \tikzset{knoten/.style={circle,fill=black,inner sep=0.7mm}}
            \node [knoten,label=right:$b$] (b) at (0,0.5) {};
            \node [knoten,label=$c$] (c) at (-0.45,0) {};
            \node [knoten,label=right:$d$] (d) at (0,-0.5) {};
            \node [knoten,label=$a$] (e) at (0.45,0) {};
            \draw[-,orange,ultra thick] (-0.2,0) to (c);
            \draw[-,orange,ultra thick] (0.2,0) to (e);
            \draw[-,very thick] (b) to (c);
            \draw[-,very thick] (b) to (e);
            \draw[-,very thick] (c) to (d);
            \draw[-,very thick] (d) to (e);

            \draw[-,orange,ultra thick] (b) to (0,0.7);
            \draw[-,orange,ultra thick] (d) to (0,-0.7);

            \draw[-,blue,ultra thick] (b) to (0,0.3);
            \draw[-,blue,ultra thick] (d) to (0,-0.3);
            \draw[-,blue,ultra thick] (-0.7,0) to (c);
            \draw[-,blue,ultra thick] (0.7,0) to (e);
        \end{tikzpicture}
    \caption{Embedding of $K_{2,2}$ in $G^{\ast}$ with visited edges coloured orange and blue}
    \label{fig:visitedEdges}
\end{figure}
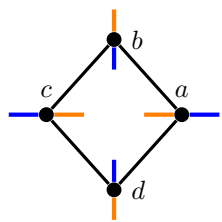
In the following, we describe a criterion to decide whether a twisted subgraph forming a cycle of even length gives rise to a strong embedding of a given graph.
\begin{lemma}\label{lemma:twistedCycle}
    Let $G$ be a 3-connected cubic planar graph and $H$ a subgraph of $G^\ast$ isomorphic to a cycle of length $2k$ with $k\geq 2$. If every pair of non-adjacent vertices $v,w\in V(H)$ satisfies $\{v,w\}\notin E(G^\ast)$, then the embedding of $G$ induced by $H$ is strong.
\end{lemma}
\begin{proof}
    Let $H$ be the cycle that is described by the closed walk $(v_1,\ldots,v_{2k})$ with $k\geq 2$. Since $H$ is an even subgraph, the embedding $\beta_G(H)$, where all edges are twisted, has exactly two facial cycles. These two facial cycles, namely $C_1$ and $C_2$, satisfy $C_1=C_2=(v_1,\ldots,v_{2k})$. We observe that these facial cycles can be distinguished from each other by observing the corresponding orientation at the vertices. In particular, this means that for $1\leq i\leq 2k$ the orientation at $v_i$ in $C_1$ is clockwise if and only if the orientation at $v_i$ in $C_2$ is anticlockwise.
    Furthermore, we see that the orientations of any pair of adjacent vertices of one of the above facial cycles alternate, i.e.\ for $i=1,\ldots,2k$ the orientation at the vertex $v_i$ in $C'\in \{C_1,C_2\}$ is clockwise if and only if the orientation of $v_{i+1}$ (with $v_{2k+1}:=v_1$) in $C'$ is anticlockwise.
    Assume that the embedding of $G$ induced by $H$ is not strong. By \cite[Theorem 19]{PaperMeikeStrong}, this means that one of the two facial cycles of $\beta_G(H)$ contains a pair of non-adjacent vertices such that the corresponding sets of visited edges with respect to $H$ are not disjoint. Let $v_i$ and $v_j$ be vertices of $H$ that share a common visited edge $e$. Thus, $e=\{v_i,v_j\}$, which implies that $v_i$ and $v_j$ are vertices that are adjacent in $G^{\ast}$ but not in $H$. Hence, the result follows.
\end{proof}
We use the above lemma to prove the following two propositions. We will exploit these propositions to establish that cyclically 4-edge connected cubic planar graphs always have strong embeddings on orientable surfaces of positive genera.
\begin{proposition}
    \label{prop:even}
    Let $G$ be a cyclically 4-edge connected cubic planar graph. If $G$ has a face of length $2k$ with $k\geq 2$, then $G$ has a strong embedding on an orientable surface of genus $k-1$.
\end{proposition}
\begin{proof}
    In the following, we read all subscripts modulo $2k$.
    Since $G$ has a face of length $2k$, we know that there exists a vertex $v\in V(G^{\ast})$ with $\deg_{G^\ast}(v)=2k\geq 4$.
    Let $v_1,\dots,v_{2k}$ be exactly the vertices adjacent to $v$ with $e_i=\{v_i,v_{i+1}\}$ for $1\leq i \leq 2k$ being edges in $G^\ast$, and $H$ the subgraph of $G^\ast$ defined by these vertices and edges. Since $H$ forms a cycle of length $2k$, it forms an even subgraph of $G^\ast$.
    Thus, the embedding defined by $H$ is orientable by \cite[Lemma 8]{EnamiEmbeddings} and it remains to show that it is strong. This is achieved by exploiting \Cref{lemma:twistedCycle}.
    For this, we assume that $v_i$ and $v_j$ are vertices in $V(H)$ satisfying $\{v_i,v_j\}\in E(G^\ast)$ and $\{v_i,v_j\}\notin E(H).$ Then $(v,v_i,v_j)$ forms a separating triangle in $G^{\ast}$ and thus $G$ cannot be cyclically 4-edge connected. Hence, the embedding induced by $H$ is strong.
    Moreover, the constructed embedding has transformed $2k$ faces of the planar embedding of $G$ into two new facial cycles. 
    Therefore, the new embedding of $G$ has an Euler characteristic of $2-2k+2=4-2k$, which means its genus is $k-1$. This concludes the proof.
\end{proof}

\begin{proposition}\label{prop:odd}
Let G be a cyclically 5-edge connected cubic planar graph. Furthermore, let $v$ and $w$ be adjacent vertices in $G^\ast$ with $\deg(v)=2k+1$ and $\deg(w)=2m+1$ for $k,m\geq 2$.
Then $G$ has a strong embedding on an orientable surface of genus $k+m-2$.
\end{proposition}
\begin{proof}
Let $v_1,\dots,v_{2k+2m-2}$ be exactly the adjacent vertices of $v$ and $w$ that are distinct from $v$ and $w$ with $e_i=\{v_i,v_{i+1}\}$ for $1\leq i \leq 2k+2m-2$ being edges in $G^\ast$, where we read the subscripts modulo $2k+2m-2$.
Furthermore, let $H$ be the subgraph defined by these vertices and the edges $\{e_1,\dots,e_n\}$.
Hence, $H$ forms an even subgraph of $G^{\ast}$ and in particular a cycle of even length. Thus, the embedding of $G$ defined by $H$ is orientable, and to show that it is a strong embedding we exploit \Cref{lemma:twistedCycle}.
For this, we assume that $v_i$ and $v_j$ are vertices in $V(H)$ satisfying $\{v_i,v_j\}\in E(G^\ast)$ and $\{v_i,v_j\}\notin E(H).$
We have to distinguish two different cases:
\begin{itemize}
    \item[1.] If $v_i$ and $v_j$ are both adjacent to a vertex $v'\in\{v,w\}$, then $(v',v_i,v_j)$ is a separating triangle in $G^{\ast}$ and thus $G$ cannot be cyclically 4-edge connected. 
    \item[2.] If $v_i$ is adjacent to $v$ and $v_j$ to $w$, then $(v,v_i,v_j,w)$ is a separating cycle of length four in $G^{\ast}$ and therefore $G$ cannot be cyclically 5-edge connected. 
\end{itemize}
Thus, the embedding of $G$ induced by $H$ is strong.  Moreover, the Euler characteristic of the new embedding of $G$ is $2-(2k+2m-2)+2=6-2k-2m$, which means that its genus is $k+m-2$.
\end{proof}

\begin{theorem}\label{corollary:orientable}
    Let $G$ be a cubic planar graph. If $G$ is cyclically 4-edge connected, then $G$ has a strong embedding on an orientable surface of positive genus.
\end{theorem}
\begin{proof}
    If $G$ is cyclically $4$-edge connected but not cyclically $5$-edge connected then $G$ has a strong embedding on the torus, as shown in~\cite[Corollary 29]{PaperMeikeStrong}. Additionally, if $G$ has a face of even length, then $G$ has a strong embedding on an orientable surface by \Cref{prop:even}. Therefore, we can assume that $G$ is cyclically $5$-edge connected and contains only faces of odd length. Since $G$ is cyclically 5-edge connected, every face of $G$ in its planar embedding is of length at least five. Hence, the result follows with \Cref{prop:odd}.
\end{proof}

\subsection{Data on 3-Connected Cubic Planar Graphs}
A $3$-connected cubic planar graph $G$ can have different inequivalent strong embeddings on surfaces of positive genera. Note that it is possible that some of these inequivalent embeddings are isomorphic. We therefore characterise the isomorphism classes of strong embeddings of $G$ on some surfaces of positive genera in this section.
These results are used to create a data base containing exactly one representatives from each isomorphism class of strong embeddings of 3-connected cubic planar graphs with up to $22$ vertices on the projective plane, the torus and the Klein bottle.
First, we relate isomorphic strong embeddings to certain orbits under $\Aut(G^\ast)$ acting on the subgraphs of $G^\ast.$
\begin{lemma}\label{lemma:isomorphic}
    Let $G$ be a $3$-connected cubic planar graph and $H_1$, and $H_2$ twisted subgraphs of $G^{\ast}$ inducing strong embeddings of $G$.
    The two strong embeddings of $G$ defined by $H_1$ and $H_2$ are isomorphic if and only if $H_1$ and $H_2$ lie in the same $\Aut(G^{\ast})$-orbit.
\end{lemma}
\begin{proof}
    Let $\beta_1(G)$ and $\beta_2(G)$ be the strong embeddings of $G$ defined by $H_1$ and $H_2$, respectively. Moreover, let $\mathcal{T}_1$ and $\mathcal{T}_2$ be the edges of $H_1$ and $H_2$, where the dual edges of $\mathcal{T}_1$ and $\mathcal{T}_2$ are denoted by $\mathcal{T}^\ast_1$ and $\mathcal{T}^\ast_2$. Additionally, let $\rho$ be the rotation system of the planar embedding of $G$.
    
    First, we assume that $H_1$ and $H_2$ lie in the same orbit under $\Aut(G^{\ast})$. Thus, there exists an automorphism $\phi^{\ast}\in\Aut(G^{\ast})$ such that $\phi^{\ast} (V(H_1))=V(H_2)$ and $\phi^{\ast} (\mathcal{T}_1)=\mathcal{T}_2$. This means that $\phi^{\ast}$ induces an automorphism $\phi$ of $G$ with $\phi(\mathcal{T}^\ast_1)=\mathcal{T}^\ast_2$. Furthermore, the image of each face of the planar embedding of $G$ under $\phi$ has to be a face of $G$.
    Thus, for all edges $e_1,e_2\in E(G)$ both incident to $v\in V(G)$ with $\rho_v(e_1)=e_2$, we know that $\rho_{\phi(v)}(\phi(e_1))=\phi(e_2)$ holds. This implies that $\phi$ preserves the rotation system $\rho$ of $G$.
    Since $\phi(\mathcal{T}^\ast_1)=\mathcal{T}^\ast_2$ it follows that $\phi$ maps the facial cycles of $\beta_1(G)$ to the facial cycles of $\beta_2(G)$.\\
    Now, let $\phi\in \Aut(G)$ be an automorphism that maps the facial cycles of $\beta_1(G)$ to the facial cycles of $\beta_2(G)$. Thus, the rotation system of $G$ is preserved by $\phi$. This implies $\phi(\mathcal{T}^\ast_1)=\mathcal{T}^\ast_2$ and therefore there exists an automorphism $\phi^\ast\in \Aut(G^\ast)$ that satisfies $\phi^{\ast}(V(H_1))=V(H_2)$ and $\phi^{\ast} (\mathcal{T}_1)=\mathcal{T}_2$.
\end{proof}
We obtain the following two theorems as a direct consequence of the above lemma and the characterisation of Weiß et al.\ \cite[Theorem 1]{PaperMeikeStrong}.  
\begin{theorem}\label{theorem:nonIsomorphic}
    Let $G$ be a 3-connected cubic planar graph. 
    There exists a one-to-one correspondence between the isomorphism classes of strong embeddings of $G$ on
    \begin{itemize}
        \item[1)] the projective plane and the $\Aut(G^{\ast})$-orbits of subgraphs of $G^{\ast}$ that are isomorphic to $K_4$.
        \item[2)] the torus and the $\Aut(G^{\ast})$-orbits of subgraphs of $G^{\ast}$ that are isomorphic to $K_{2,2,2}$ or $K_{2,2m}$ for $m\geq 1$, where for $K_{2,2m}$ the vertices in the partition sets of size two are not adjacent in $G^{\ast}$.
        \item[3)] the Klein bottle and the $\Aut(G^{\ast})$-orbits of subgraphs of $G^{\ast}$ that are isomorphic to $A_3, A_5, A_6$ or $K_{2,2m-1}$ for $m\geq 2$, where for $K_{2,2m-1}$ the vertices in the partition sets of size two are not adjacent in $G^{\ast}$.
    \end{itemize}
\end{theorem}

\begin{theorem}
    Let $G$ be a 3-connected cubic planar graph. 
    \begin{itemize}
        \item[1)] If $G$ is cyclically 4-edge connected, then the $\Aut(G)$-orbits of faces of length $2k$ for $k\geq 2$ induce non-isomorphic strong embeddings on orientable surfaces of genus $k-1$.
        \item[2)] If $G$ is cyclically 5-edge connected, then the $\Aut(G)$-orbits of two neighbouring faces of length $2k+1$ and  $2m +1$ for $k,m\geq 2$ induce non-isomorphic strong embeddings on orientable surfaces of genus $m+k-2$.
    \end{itemize}
\end{theorem}

These theorems allow us to employ GAP \cite{GAP4} in order to compute representatives of the isomorphism classes of strong embeddings of $3$-connected cubic planar graphs on up to $22$ vertices on the projective plane, the torus or the Klein bottle. In order to present our compiled data in \Cref{table:Embeddings}, we denote the number of isomorphism classes of strong embeddings of 3-connected cubic planar graphs with $n$ vertices on the projective plane, the torus, and the Klein bottle by $p_n$, $t_n$ and $k_n$, respectively.
\begin{table}[H]
    \centering
    \begin{tabular}[h]{|c|c|c|c|c|c|c|c|c|c|c|c|}
    \hline
    \textbf{$n$} & \textbf{4} &  \textbf{6}& \textbf{8} &\textbf{10} &\textbf{12} & \textbf{14}&\textbf{16}&\textbf{18}&\textbf{20}&\textbf{22}\\
     \hline
     $p_n$ & 1& 1& 2& 7& 31& 152& 917& 5914& 40633 & 290295 \\
     \hline
     $t_n$ & 0& 0& 3& 5& 33& 156& 1084& 7456& 55885 & 425789 \\
     \hline
     $k_n$ & 0& 1& 4& 19& 103& 629& 4253& 30505& 226806 & 1626611\\
     \hline
\end{tabular}
\caption{Numbers of isomorphism classes of strong embeddings on the projective plane ($p_n$), the torus ($t_n$) and the Klein bottle ($k_n$) of 3-connected cubic planar graphs with $n$ vertices}
\label{table:Embeddings}
\end{table}

In \cite{PaperMeikeStrong}, Weiß et al.\ have examined the $3$-connected cubic planar graphs that have at least one strong embedding on the projective plane, the torus and the Klein bottle. Looking into the data of 3-connected cubic planar graphs that have no strong embedding on one of these surfaces turns out to be very intriguing. 
For this, let $G_n$ be the set of 3-connected cubic planar graphs with $n$ vertices and $g_n:=\vert G_n\vert$. Moreover, let $p_n', t_n'$ and $k_n'$ be the numbers of graphs in $G_n$ that do not have a strong embedding on the projective plane, the torus, and the Klein bottle, respectively.

\begin{table}[H]
    \centering
    \begin{tabular}[h]{|c|c|c|c|c|c|c|c|c|c|c|c|}
    \hline
    \textbf{$n$} & \textbf{4} &  \textbf{6}& \textbf{8} &\textbf{ 0} &\textbf{12} & \textbf{14}&\textbf{16}&\textbf{18}&\textbf{20}&\textbf{22}\\
    \hline
     $ g_n$ & 1& 1& 2& 5& 14& 50& 233& 1249& 7595 & 49566\\
     \hline
     $p_n'$ & 0& 0& 1& 1& 2& 5& 11& 30& 110& 417\\
     \hline
     $t_n'$ & 1& 1& 1& 3& 7& 24& 93& 434& 2111 &11003\\
     \hline
     $k_n'$ & 1& 0& 0& 0& 0& 0& 0& 0& 1& 153\\
     \hline
\end{tabular}
\caption{Numbers of 3-connected cubic planar graphs with $n$ vertices that have no strong embedding on the projective plane, the torus or the Klein bottle}
\label{table:NoEmbedding}
\end{table}
We observe that there are exactly two graphs with fewer than $22$ vertices that do not have a strong embedding on the Klein bottle.
The two graphs are the $K_4$ and the dodecahedral graph.
The latter is due to the fact that the dodecahedral graph is the smallest cyclically 5-edge-connected graph, as mentioned in \cite{PaperMeikeStrong}.
The data for the torus in the table indicates that the $3$-connected cubic planar graphs with at most 22 vertices and no strong embedding on the torus belong to two distinct classes.
As shown in \cite{PaperMeikeStrong}, the first class contains the cyclically 5-edge connected graphs or graphs that can be obtained from them by recursively replacing a vertex by a triangle. This is why the dodecahedral graph ($20$ vertices) and the graph obtained by replacing a vertex in the dodecahedral graph with a triangle ($22$ vertices) have no strong embeddings on the torus.
Surprisingly, the remaining 3-connected cubic planar graphs with at most 22 vertices and no strong embedding on the torus are Apollonian duals. For better illustration, we present the numbers $t_n'$ together with the cardinalities of Apollonian duals with exactly $n$ vertices, denoted by $a_n$.

\begin{table}[H]
    \centering
    \begin{tabular}[h]{|c|c|c|c|c|c|c|c|c|c|c|c|}
    \hline
    \textbf{$n$} & \textbf{4} &  \textbf{6}& \textbf{8} &\textbf{10} &\textbf{12} & \textbf{14}&\textbf{16}&\textbf{18}&\textbf{20}&\textbf{22}\\
    \hline
     $a_n$ & 1& 1& 1& 3& 7& 24& 93& 434& 2110& 11002\\
  \hline
          $t_n'$ & 1& 1& 1& 3& 7& 24& 93& 434& 2111 &11003\\
     \hline
\end{tabular}
\caption{Numbers of Apollonian duals and number of $3$-connected cubic planar graphs with $n$ vertices that have no strong embedding on the torus}
\label{table:Apollonian}
\end{table}

However, considering strong embeddings on other orientable surfaces, we see that the dodecahedral graph has strong embeddings on orientable surfaces with Euler characteristic $-2$ and $-4$. Therefore, the graph resulting from replacing a vertex of the dodecahedral graph with a triangle also has strong embeddings on these surfaces. By first experimenting in GAP, we verified that no Apollonian dual with fewer than $22$ vertices has a strong embedding on an orientable surface of positive genus.
Indeed, we have been able to show that Apollonian duals cannot be strongly embedded on orientable surfaces of positive genera. This is detailed in \Cref{theorem:noOrient} and begs up the question whether Apollonian duals are precisely the 3-connected cubic planar graphs that have no strong embedding on any orientable surface of positive genus. The answer to this question is the focus of \Cref{sec:orient}.

\subsection{Exponential many Non-Isomorphic Strong Embeddings}
We know that a 3-connected cubic planar graph $G$ can have at most $p:=\frac{|V(G)|-2}{2}$ non-isomorphic strong embeddings on the projective plane, see \cite[Corollary 31]{PaperMeikeStrong}. This means that the number of non-isomorphic strong embeddings on the projective plane grows linearly in the number of vertices.
Moreover, in \cite{EnamiEmbeddings}, Enami presents a family of $3$-connected cubic planar graphs where the number of inequivalent embeddings on the torus and the Klein bottle grows exponentially in the number of vertices. However, these embeddings are not strong. Thus, we strengthen Enami's result by constructing a family of $3$-connected cubic planar graphs whose number of non-isomorphic strong embeddings on the torus and the Klein bottle, grows exponentially in the number of vertices.

\begin{theorem}\label{theorem:exponential}
    Let $n>3$ be a natural number and $i:= n \pmod 2.$ Then there exists a $3$-connected cubic planar graph $G$ that has at least $\binom{2n}{n-i}$ non-isomorphic strong embeddings on the torus and  at least $\binom{2n}{n-1+i}$ non-isomorphic strong embeddings on the Klein bottle.
\end{theorem}
\begin{proof}
    In order to prove the above statement, we construct a family of planar triangulations $\mathcal{T}=\{T_n\mid n>3\}$. Our aim is to construct the family $\mathcal{T}$ such that every triangulation $T_n\in\mathcal{T}$ satisfies $\vert V(T_n)\vert=2n+5$ and $\vert \Aut(T_n)\vert =1.$ In particular, we construct the planar triangulation $T_n$ such that for $j\in\{0,1\}$ the graph $T_n$
    has exactly $\binom{2n}{n-j}$ subgraphs that are isomorphic to $K_{2,n-j}$, where the vertices of the partition sets of size two are not adjacent in $T_n$. 
    Since the dual graph of $T_n$ is a 3-connected cubic planar graph, \Cref{theorem:nonIsomorphic} implies for $i:= n \pmod 2$ that $({T_n})^\ast$ has at least $\binom{2n}{n-i}$ non-isomorphic strong embeddings on the torus and  at least $\binom{2n}{n-1+i}$ non-isomorphic strong embeddings on the Klein bottle. For simplicity, we first present an illustration of this planar triangulation $T_n$ in \Cref{fig:double2ngontet} and formalise its construction in the following.

    In order to define $T_n,$ we introduce a planar triangulation $T$ as the planar graph that contains a cycle $(v_1,\dots,v_{2n})$ of length $2n$, along with two additional vertices, namely $v_{2n+1}$ and $v_{2n+2}$, which are both adjacent to all vertices $v_1,\dots,v_{2n}$.
    We obtain the desired triangulation $T_n$ by manipulating the incidence structure of $T$, which is achieved by employing the subdivision of a triangle from \Cref{def:subdivision}:
\begin{enumerate}
    \item We subdivide the face $(v_1,v_2,v_{2n+1})$ of $T$ and get a planar triangulation $T'$ and a vertex $w_1.$
    \item The face $(v_1,v_2,w_1)$ of $T'$ is subdivided to get a planar triangulation $T''$ and a vertex $w_2.$
    \item Finally, we subdivide the face $(v_2,v_3,v_{2n+1})$ of $T''$ and obtain $T_n$ with a vertex $w_3.$
\end{enumerate}
Hence, $T_n$ is a planar triangulation with $2n+5$ vertices.
We observe that every subset of $\{v_1,\ldots,v_{2n}\}$ of cardinality $m\in\{n,n-1\}$ induces to a graph $K_{2,m},$ where the vertices of the partition set of size two, namely $v_{2n+1}$ and $v_{2n+2}$, are not adjacent in $T_n.$ Thus, it remains to show that the automorphism group of $T_n$ is trivial to conclude the result.
To see this, we observe that the vertices of $T_n$ satisfy
\begin{align*}
&\deg(v_1)=6,\,\deg(v_2)=7,\,\deg(v_3)=5,\,\deg(v_{2n+1})=2n+2,\\    
&\deg(v_{2n+2})=2n,\,\deg(w_1)=4,\,\deg(w_2)=\deg(w_3)=3
\end{align*}
and $\deg(v_{i})=4$ for all $i=4,\ldots,2n.$
Let $\phi \in \Aut(T_n)$ be an automorphism. We have to show $\phi(v)=v$ for all $v\in V(T_n).$ First, we notice that $\phi$ is a bijective map that maps a vertex to a vertex of the same degree. 
Observe that, since $n\geq 4$, the vertices $v_1, v_2, v_3, v_{2n+1}, v_{2n+2}$ are the unique vertices of their respective degrees. 
Hence, $\phi(v_1)=v_1,\,\phi(v_2)=v_2,\,\phi(v_3)=v_3,\,\phi(v_{2n+1})=v_{2n+1}$ and $\phi(v_{2n+2})=v_{2n+2}$ follows. Further, $\phi(w_3)$ has to be a vertex of degree $3$ that is incident to $\phi(v_3)=v_3.$
This implies $\phi(w_3)=w_3$ which in turn implies $\phi(w_2)=w_2$ and $\phi(w_1)=w_1.$
Lastly, $\phi(v_i)=v_i$ for all $i=4,\ldots,2n$ follows from the fact that $(\phi(v_1),\ldots,\phi(v_{2n}))=(v_1,v_2,v_3,\phi(v_4),\ldots,\phi(v_{2n}))$ has to be a cycle of length $2n,$ namely the unique cycle $(v_1,\ldots,v_{2n}).$ Hence, $\Aut(T_n)$ is trivial and the result follows.
\end{proof}
       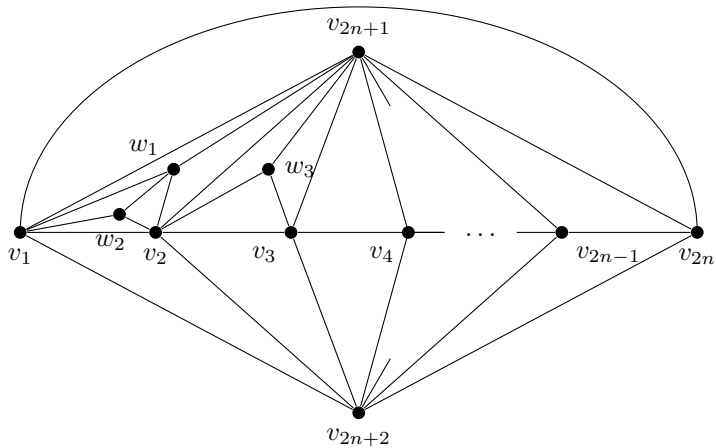
\begin{figure}[H]
        \centering
        \begin{tikzpicture}[scale=1.2]
            \tikzset{knoten/.style={circle,fill=black,inner sep=0.6mm}}
            \node [knoten,label=below :$v_{2n}$] (a) at (7.5,0) {};
            \node [knoten,label=below right :$v_{2n-1}$] (b) at (6,0) {};

            \node [knoten,label=below left:$v_3$] (d) at (3,0) {};
            \node [knoten,label=below :$v_2$] (e) at (1.5,0) {};
            \node [knoten,label=below :$v_1$] (f) at (0,0) {};
            \node [knoten,label=above:$v_{2n+1}$] (g) at (3.75,2) {};
            \node [knoten,label=below :$v_{2n+2}$] (h) at (3.75,-2) {};

            \node [knoten,label=above left:$w_1$] (v1) at (1.7,0.7) {};
            \node [knoten,label=right:$w_3$] (v2) at (2.75,0.7) {};
            \node [knoten] (v3) at (1.1,0.2) {};
            \node [label=$w_2$] (x) at (1.,-0.4) {};
                        \node [label=$\ldots$] (x) at (5.1,-0.25) {};
            \node [knoten, label= below left:$v_4$] (v111) at (4.3,0.) {};
            \draw[-] (v2) to (d);
            \draw[-] (v2) to (e);
            \draw[-] (v2) to (g);

            \draw[-] (v1) to (f);
            \draw[-] (v1) to (e);
            \draw[-] (v1) to (g);
             \draw[-] (v3) to (f);
            \draw[-] (v3) to (e);
            \draw[-] (v3) to (v1);

            \draw[-] (g) to (4.3,0);
            \draw[-] (h) to (4.3,0);
            \draw[-] (g) to (4.1,1.4);
            \draw[-] (h) to (4.1,-1.4);
                    \draw[-] (4.7,0) to (4.3,0);
            \draw[-] (a) to (b);
            \draw[-] (b) to (5.5,0);
            \draw[-] (d) to (4.7,0);

            \draw[-] (d) to (e);
            \draw[-] (e) to (f);
            
            \draw[-] (g) to (a);
            \draw[-] (g) to (b);
            \draw[-] (g) to (d);
            \draw[-] (g) to (e);
            \draw[-] (g) to (f);
            \draw[-] (h) to (a);
            \draw[-] (h) to (b);
            \draw[-] (h) to (d);
            \draw[-] (h) to (e);
            \draw[-] (h) to (f);
            \draw (0,0) arc
            [
                start angle=180,
                end angle=0,
                x radius=3.75cm,
                y radius =2.5cm
            ] ;
        \end{tikzpicture}
         \caption{The planar triangulation $T_n$}
        \label{fig:double2ngontet}
    \end{figure}

\section{Strong Embeddings of Apollonian Duals}\label{subsec:apollonian}
In this section, we study strong embeddings of Apollonian duals on orientable surfaces, the projective plane and the Klein bottle. Moreover, we present upper and lower bounds on the number of isomorphism classes of strong embeddings of Apollonian duals on the projective plane and the torus.
Before stating these results we give a note on the number of triangles that are contained in a given Apollonian network.

\begin{remark}\label{remark:triangles}
    For a given number of vertices, the Apollonian networks are the maximal planar graphs with the largest possible number of triangles contained in the graph. 
    This means that an Apollonian network with $n$ vertices has $3n-8$ triangles, see~\cite{NumberCycles}. Since an Apollonian network with $n$ vertices has $2n-4$ faces that are triangles, we can conclude that it has $n-4$ separating triangles.
\end{remark}

Next, we recall a known fact about subgraphs of Apollonian networks which are isomorphic to $K_4$. We exploit this instance to construct strong embeddings of Apollonian duals on the projective plane.

\begin{remark}
In \cite[Lemma 15]{EnamiEmbeddings}, Enami states that an Apollonian network $T$ has $\vert V(T)\vert -3$ subgraphs that are isomorphic to $K_4$. Thus, $T^\ast$ has exactly $\vert V(T)\vert -3$ inequivalent strong embeddings on the projective plane. Moreover, $T^\ast$ has at least one and at most $\vert V(T)\vert -3$ non-isomorphic strong embeddings on the projective plane.    
\end{remark}

With~\cite{Digraphs,GAP4,simplicialsurfacegap} we are able to compute all non-isomorphic strong embeddings of Apollonian duals with a fixed number of vertices on the projective plane. The numbers of these non-isomorphic strong embeddings can be seen in Table~\ref{tab:proj}. In the described table, the cardinality of non-isomorphic strong embeddings of Apollonian duals with $n$ vertices on the projective plane is denoted by $p_n^A$.
\begin{table}[H]
    \centering
    \begin{tabular}[h]{|c|c|c|c|c|c|c|c|c|c|c|c|}
\hline
\textbf{$n$} & \textbf{ 4} &  \textbf{6}& \textbf{8} &\textbf{ 10} &\textbf{ 12} & \textbf{14}&\textbf{16}&\textbf{18}&\textbf{20}&\textbf{22}\\
\hline
 $p_n^A$ & 1& 1& 2& 6& 25& 111& 590& 3251& 18620& 108619\\
 \hline
\end{tabular}
\caption{Number of non-isomorphic strong embeddings of Apollonian duals with $n$ vertices on the projective plane}
\label{tab:proj}
\end{table}

\subsection{Strong Orientable Embeddings of Apollonian Duals}\label{subsec:orientableEmbedding}
The aim of this subsection is to prove that Apollonian duals have no strong embedding on an orientable surface of positive genus.
We obtain this result by exploiting the fact that a twisted subgraph defining an embedding on an orientable surface has to be even.
In addition, we utilise the fact that a twisted subgraph $H$ of the dual graph of a 3-connected cubic planar graph G does not define a strong embedding if $\deg_H(v)=2$ and $\deg_{G^\ast}(v)=3$ for some vertex $v\in V(H)$, see \cite[Lemma 12]{PaperMeikeStrong}.

\begin{theorem}\label{theorem:noOrient}
An Apollonian dual has no strong embedding on an orientable surface of positive genus.
\end{theorem}
\begin{proof}
    We prove the statement by contradiction and assume that there exists an Apollonian dual $G$ that can be strongly embedded on an orientable surface of positive genus. We further assume that $G$ is minimal with respect to the number of vertices.
    Since $G$ can be strongly embedded on an orientable surface of positive genus, the Apollonian network $G^{\ast}$ contains an even subgraph $H$ inducing the described strong embedding. 
    The smallest Apollonian network, namely $K_4$, has exactly two strong embeddings, namely its unique planar embedding and an embedding on the projective plane. Consequently, $K_4$ has no strong embedding on an orientable surface of positive genus, which means $G^\ast\ncong K_4$.
    We can therefore assume that $\vert V(G^{\ast})\vert \geq 5$ and that $G^{\ast}$ is constructed by subdividing at least one triangle of $K_4$. So there exists a vertex $v\in V(G^{\ast})$ with $\deg_{G^{\ast}}(v)=3$ such that the deletion of $v$ results in ${(G^{\ast})}_v$ being a well-defined Apollonian network (see \Cref{def:subdivision}). Now, we give a case distinction whether the vertex $v$ is contained in the even subgraph $H$ or not. First, we assume that $v$ is not in $H$. This implies that the even subgraph $H$ of $G^{\ast}$ is fully contained in ${(G^{\ast})}_v$. By \cite[Lemma 8]{EnamiEmbeddings}, this means that the dual graph of ${(G^{\ast})}_v$ can be strongly embedded on an orientable surface of positive genus which contradicts the minimality of $G$.
    Hence, the vertex $v$ has to be contained in $H$. Since $H$ is even, this implies $\deg_H(v)=2$. However, since $\deg_{G^\ast}(v)=3$, the embedding of $G$ defined by $H$ cannot be strong. We therefore obtain the desired contradiction, and the result follows.
\end{proof}

\subsection{Strong Embeddings of Apollonian Duals on the Klein bottle}\label{sec:kleinbottle}
Next, we establish the existence of strong embeddings of Apollonian duals on the Klein bottle and derive lower bounds on the number of non-isomorphic strong embeddings on the Klein bottle.
\Cref{theorem:reembedding} asserts that in order to strongly embed an Apollonian dual on the Klein bottle we have to find subgraphs in the corresponding Apollonian network that are isomorphic to the graphs $A_3,A_5,A_6$ (see \Cref{fig:subgraphsKleinBottle}) or the graph $K_{2,2m-1}$ for $m\geq 2$, where the two vertices in the partition set of size two are not adjacent.
The following lemma shows that an Apollonian network that is not isomorphic to $K_4$ always contains one of those subgraphs that can be exploited to construct a strong embedding of the given Apollonian dual on the Klein bottle.

\begin{lemma}\label{prop:numTrian}
Every separating triangle of an Apollonian network $T$ gives rise to a subgraph that is isomorphic to $K_{2,3}$, where the vertices in the partition set of size two are not adjacent in $T$.
\end{lemma}
\begin{proof}
    Let $T$ be an Apollonian network not isomorphic to $K_4$. Hence, $T$ is constructed from $K_4$ by recursively subdividing faces. If the vertices of a face $(v_1,v_2,v_3)$ that is subdivided in the process of constructing $T$ are adjacent to another common vertex $v$, then these four vertices together with the new vertex $w$ give rise to a complete bipartite graph $K_{2,3}.$ Here, $\{v,w\}$ and $\{v_1,v_2,v_3\}$ form the partition sets, and $v$ and $w$ are not adjacent, since $T$ is planar.
    After the subdivision, the cycle $(v_1,v_2,v_3)$ forms a separating triangle of the constructed Apollonian network $T$. Since each set of three vertices of a subgraph of $T$ isomorphic to $K_4$ are adjacent to another common vertex, there is a bijection between the separating triangles of an Apollonian network $T$ and the subgraphs of $T$ being isomorphic to $K_{2,3}$, where the vertices in the partition set of size two are not adjacent in $T$. 
\end{proof}

Since every Apollonian network other than $K_4$ contains a separating triangle (see \Cref{remark:triangles}), the above result implies the existence of strong embeddings of Apollonian duals on the Klein bottle by \Cref{theorem:reembedding}.

\begin{corollary}\label{theorem:kleinbottle}
    Each Apollonian dual not isomorphic to $K_4$ has a strong embedding on the Klein bottle.
\end{corollary}

In the following, we show that if an Apollonian network contains a subgraph that is isomorphic to $K_{2,2m-1}$, then the equality $m=2$ must be satisfied.

\begin{proposition}\label{nok22m-1}
    An Apollonian network $T$ contains no subgraph isomorphic to $K_{2,2m-1}$ for $m\geq 3$, where the vertices in the partition set of size two are not adjacent in $T$.
\end{proposition}
\begin{proof}
    We assume that $T$ is an Apollonian network with a subgraph $H$ is isomorphic to $K_{2,2m-1}$ for $m\geq 3$ and that $T$ is minimal with respect to the number of vertices. Since the complete graph $K_4$ does not contain a subgraph isomorphic to $K_{2,2m-1}$, we know that $\vert V(T)\vert\geq 5$. Hence, $T$ is constructed by applying at least one subdivision to $K_4$. Thus, there is a vertex $v\in V(T)$ with $\deg_T(v)=3$, so the graph $T_v$ obtained by deleting $v$ is a well-defined Apollonian network (see \Cref{def:subdivision}). We make a distinction based on whether $v$ is contained in $H$ or not.
    \begin{itemize}
    \item[(a)] If $v$ is not contained in the subgraph $H$, then $H$ is fully contained in $T_v$. This means that $T_v$ contains a subgraph isomorphic to $K_{2,2m-1},$ where the vertices in the partition set of size two are not adjacent which contradicts the minimality of $T$.
    \item[(b)] Furthermore, if $v$ is contained in $H$, then the following contradiction arises: On the one hand, $v$ cannot be a vertex of the partition set of size $2m-1$, since the three adjacent vertices of $v$ are pairwise adjacent, thus any two of these vertices cannot form the described partition set of size 2. On the other hand, $v$ cannot be a vertex of the partition set of size 2, because $v$ has exactly three neighbours and therefore cannot be adjacent to $2m-1$ vertices of $T$ for $m\geq 3$.
\end{itemize}
\end{proof}

In the following, we use the above statements to give a remark on the number of possible strong embeddings of a given Apollonian dual on the Klein bottle.
\begin{remark}
With \Cref{prop:numTrian} and the fact that an Apollonian network $G^{\ast}$ has exactly $|V(G^{\ast})|-4$ separating triangles, we obtain at least $|V(G^{\ast})|-4$ inequivalent strong embeddings on the Klein bottle.  
By denoting the set of all separating triangles of $G^{\ast}$ by $\mathcal{W}$ we conclude that the number of non-isomorphic strong embeddings of $G$ on the Klein bottle is at least $\vert \mathcal{W}^{\Aut(G^{\ast})}\vert$ where $\mathcal{W}^{\Aut(G^{\ast})}$ denotes the set of orbits of $Aut(G^{\ast})$ on the set $\mathcal{W}$.
\end{remark}

Again, using GAP and the packages Digraphs and Simplicial Surfaces~\cite{Digraphs,GAP4,simplicialsurfacegap} we compute all isomorphism classes of strong embeddings of Apollonian duals with a fixed number of vertices on the Klein bottle. The numbers of these non-isomorphic strong embeddings are recorded in Table~\ref{tab:klein}. In this table, the number of isomorphism classes of strong embeddings of Apollonian duals with $n$ vertices on the Klein bottle is denoted by $k^A_n$.
\begin{table}[H]
    \centering
    \begin{tabular}[h]{|c|c|c|c|c|c|c|c|c|c|c|c|}
\hline
\textbf{$n$} & \textbf{ 4} &  \textbf{6}& \textbf{8} &\textbf{ 10} &\textbf{ 12} & \textbf{14}&\textbf{16}&\textbf{18}&\textbf{20}&\textbf{22}\\
\hline
 $k^A_n$ & 0& 1& 2& 10& 47& 283& 1750& 11388& 74334& 488785\\
 \hline
\end{tabular}
\caption{Number of non-isomorphic strong embeddings of Apollonian duals with $n$ vertices on the Klein bottle}
\label{tab:klein}
\end{table}

\section{\texorpdfstring{Strong Orientable Embeddings of $3$-Connected Cubic Planar Graphs}{Strong Orientable Embeddings of 3-Connected Cubic Planar Graphs}}
\label{sec:orient}
In this section we prove that the Apollonian duals are exactly the $3$-connected cubic planar graphs that have no strong embedding on any orientable surface of positive genus, see \Cref{theorem:mainNoOrient}. 
The proof relies on the fact that any 3-connected cubic planar graph can be uniquely decomposed into graphs that are either isomorphic to $K_4$ or cyclically 4-edge connected, see ~\cite[Theorem 3.5]{fouquet}. In order to introduce this decomposition of a $3$-connected cubic graph, we have to define the splitting of a cyclic edge cut. 

\begin{definition}\label{def:splitting}
Let $G$ be a $3$-connected cubic graph which is not cyclically $4$-edge connected and $S=\{\{v_1,w_1\},\{v_2,w_2\},\{v_3,w_3\}\}\subseteq E(G)$ a cyclic edge cut of size three. Deleting the edges of $S$ results in the disconnected graph $G'$ with connected components $G'_1$ and $G'_2$. Without loss of generality we assume that $v_1,v_2,v_3$ are contained in $G'_1$ and $w_1,w_2,w_3$ are contained in $G'_2,$ hence these vertices all have degree two in $G'$. So by adding two new vertices $v$ and $w$ with $v,w\notin V(G)$ and the edges $\{v,v_i\}$ and $\{w,w_i\}$ for $i=1,2,3$ we have constructed a cubic graph with two connected components.
This operation is called the \textbf{\emph{splitting}} of a cyclic edge cut of size three.
\end{definition}

To clarify this operation, consider the 3-connected cubic planar graph shown in \Cref{fig:octtet}. There, the set $\{\{v_i,w_i\}\mid i=1,2,3\}$ is a cyclic edge cut of size three. Applying the splitting operation to this cut yields the disconnected graph depicted in \Cref{fig:decomp}.

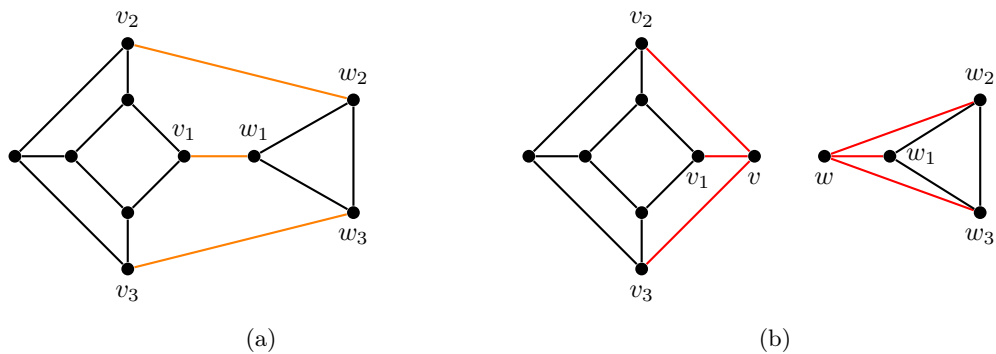
\begin{figure}[H]
    \centering
    \begin{subfigure}{.45\textwidth}
        \begin{tikzpicture}[scale=1.5]
    \tikzset{knoten/.style={circle,fill=black,inner sep=0.6mm}}
    \node [knoten] (a) at (0,0.5) {};
    \node [knoten, label=above:$v_1$] (b) at (0.5,0) {};
    \node [knoten] (c) at (0,-0.5) {};
    \node [knoten] (d) at (-0.5,0.) {};
    
    \node [knoten, label=$v_2$] (e) at (0,1) {};
    \node [knoten, label=$w_1$] (f) at (1.12+0.,0) {};
    \node [knoten,label=below:$v_3$] (g) at (0,-1) {};
    \node [knoten] (h) at (-1,0) {};
    
    \node [knoten, label=$w_2$] (d1) at (2+0.,0.5) {};
        \node [knoten, label=below:$w_3$] (d2) at (2+0.,-0.5) {};
    \draw[-, thick] (a) to (b);
    \draw[-, thick] (b) to (c);
    \draw[-, thick] (c) to (d);
    \draw[-, thick] (a) to (d);
    \draw[-, thick] (a) to (e);
    \draw[-, thick, orange] (b) to (f);
    \draw[-, thick] (c) to (g);
    \draw[-, thick] (d) to (h);

    \draw[-, thick] (d1) to (f);
    \draw[-, thick] (d2) to (f);
    \draw[-, thick] (d2) to (d1);
    \draw[-, thick,orange] (d1) to (e);
    \draw[-, thick,orange] (d2) to (g);
    \draw[-, thick] (h) to (g);
    \draw[-, thick] (e) to (h);

\end{tikzpicture}
        \caption{}
        \label{fig:octtet}
    \end{subfigure}
\begin{subfigure}{.45\textwidth}
        \begin{tikzpicture}[scale=1.5]
    \tikzset{knoten/.style={circle,fill=black,inner sep=0.6mm}}
    \node [knoten] (a) at (0,0.5) {};
    \node [knoten, label=below:$v_1$] (b) at (0.5,0) {};
    \node [knoten] (c) at (0,-0.5) {};
    \node [knoten] (d) at (-0.5,0.) {};
    
    \node [knoten, label=$v_2$] (e) at (0,1) {};
    \node [knoten, label=below:$w$] (f) at (1.12+0.5,0) {};
    \node [knoten,label=below:$v_3$] (g) at (0,-1) {};
    \node [knoten] (h) at (-1,0) {};
    
    \node [knoten, label=$w_2$] (d1) at (3,0.5) {};
    \node [knoten, label=below:$w_3$] (d2) at (3,-0.5) {};

    \node [knoten, label=right:$w_1$] (w) at (2.2,0) {};

    \node [knoten, label=below:$v$] (v) at (1,0) {};
    \draw[-, thick] (a) to (b);
    \draw[-, thick] (b) to (c);
    \draw[-, thick] (c) to (d);
    \draw[-, thick] (a) to (d);
    \draw[-, thick] (a) to (e);
    \draw[-, thick] (c) to (g);
    \draw[-, thick] (d) to (h);

    \draw[-, thick,red] (d1) to (f);
    \draw[-, thick,red] (d2) to (f);
    \draw[-, thick] (d2) to (d1);
    
    \draw[-, thick] (d1) to (w);
    \draw[-, thick] (d2) to (w);
    \draw[-, thick,red ] (f) to (w);

    \draw[-, thick,red] (g) to (v);
    \draw[-, thick,red] (e) to (v);
    \draw[-, thick,red] (b) to (v);
    
    \draw[-, thick] (h) to (g);
    \draw[-, thick] (e) to (h);

\end{tikzpicture}
        \caption{}
        \label{fig:decomp}
    \end{subfigure}
    \caption{A cubic planar graph (a) and the graph that is constructed by splitting the cyclic edge cut of size three in orange, where the new edges are coloured in red (b)}
    \label{fig:split}
\end{figure}

Note that the disconnected graph in \Cref{fig:decomp} is completely decomposed into components that are either isomorphic to $K_4$ or cyclically 4-edge connected. Such a unique decomposition is defined in the following.

\begin{definition}
Let $G$ be a $3$-connected cubic planar graph. We call the tuple $(G_1,\ldots,G_k)$ a \emph{\textbf{$\C4C$-decomposition}} of the graph $G$, if
    \begin{itemize}
    \item $G_1,\ldots,G_k$ are the connected components that are obtained by recursively splitting all cyclic edge cuts of size three starting with $G$ and
    \item for $1\leq i\leq k$ the graph $G_i$ is either isomorphic to $K_4$ or cyclically $4$-edge connected.
    \end{itemize}
We call the connected components $G_1,\ldots,G_k$ the \emph{\textbf{indecomposables}} of $G$. If $G$ is either cyclically 4-edge connected or isomorphic to $K_4$ then $G$ itself is the only indecomposable of $G$.
\end{definition}

Note that if an indecomposable of a 3-connected cubic planar graph $G$ has a strong embedding on a surface $S$, then $G$ also has a strong embedding on $S$. This follows from the fact that the dual graph of an indecomposable is isomorphic to a subgraph of $G^{\ast}$.
Using the result established by Fouquet et al.\ in~\cite[Theorem 3.8]{fouquet} allows us to give the following characterisation of an Apollonian network.
\begin{proposition}\label{prop:indecomposables}
Let $G$ be a 3-connected cubic planar graph. The indecomposables of the $\C4C$-decomposition are all isomorphic to the graph $K_4$ if and only if $G$ is an Apollonian dual.
\end{proposition}

This proposition helps us to prove the following theorem.
\begin{theorem}\label{theorem:mainNoOrient}
    Let $G$ be a 3-connected cubic planar graph. Then $G$ is an Apollonian dual if and only if $G$ cannot be strongly embedded on an orientable surface of positive genus.
\end{theorem}
\begin{proof}
    In \Cref{theorem:noOrient} we have established that an Apollonian dual has no strong embedding on any orientable surface of positive genus. Thus, it remains to show that there exist no other 3-connected cubic planar graph satisfying this property.
    So, let $G$ be a minimal counterexample to this statement. This means that $G$ is a 3-connected cubic planar graph not isomorphic to an Apollonian dual without any strong embedding on an orientable surface of positive genus. In particular, we assume $G$ to be minimal with respect to the number of vertices.
    Suppose that $G$ is not cyclically 4-edge connected. Hence, by \cite[Theorem 3.8]{fouquet}, $G$ has a unique $\C4C$-decomposition $(G_1,\ldots,G_k)$ with $k>1$.
    Since $G$ is not an Apollonian dual, there exists an $1\leq i\leq k$ such that $G_i$ is not isomorphic to $K_4$ by \Cref{prop:indecomposables}. 
    We observe that $G_i$ has fewer vertices than $G$ and is $3$-connected, cubic, planar and not isomorphic to an Apollonian dual. By our assumption that $G$ is a minimal counterexample, $G_i$ has a strong embedding on an orientable surface of positive genus. As discussed, a strong embedding of $G_i$ on an orientable surface of positive genus can be translated into a strong embedding of $G$ on the same surface. This is a contradiction, which proves that $G$ is cyclically 4-edge connected. Hence, $G$ has a strong embedding on an orientable surface of positive genus by \Cref{corollary:orientable}. Thus the result follows by contradiction.
\end{proof}

\section*{Acknowledgements}
We gratefully acknowledge the funding by the Deutsche Forschungsgemeinschaft (DFG, German Research Foundation) in the framework of the Collaborative Research Centre CRC/TRR 280 “Design Strategies for Material-Minimized Carbon Reinforced Concrete Structures – Principles of a New Approach to Construction” (project ID 417002380). Furthermore, R.\ Akpanya was supported by a grant from the Simons Foundation (SFI-MPS-Infrastructure-00008650).

\bibliographystyle{plain}
\bibliography{main}

\end{document}